\documentclass[reqno]{amsart}
\usepackage{fullpage}
\usepackage{array}
\usepackage{mathrsfs}
\usepackage{hyperref}
\usepackage{graphicx,color}
\usepackage{float}
\usepackage{tikz-cd}

\newcommand{\varied}

\newcommand{\var}{\ensuremath{\textnormal{var}}}

\newcommand{\F}{\mathbb{F}}
\newcommand{\V}{\mathfrak{V}}
\newcommand{\Ann}{\mathrm{Ann}}
\newcommand{\cA}{\mathcal{A}}   
\newtheorem{theorem}{Theorem}[section]
\newtheorem{lemma}[theorem]{Lemma}
\newtheorem{corollary}[theorem]{Corollary}

\newtheorem{proposition}[theorem]{Proposition}
\theoremstyle{definition}
\newtheorem{remark}[theorem]{Remark}
\newtheorem{definition}[theorem]{Definition}
\newcolumntype{C}[1]{>{\centering\let\newline\\\arraybackslash\hspace{0pt}}m{#1}}

\usepackage[utf8]{inputenc}
\usepackage{amsfonts}
\usepackage{amssymb}
\usepackage{verbatim}

\numberwithin{equation}{section}
\begin{document}
\title{Finite basis property for finite graded algebras}

%\author{}
%\thanks{{\it E-mail addresses:} }
\author[Y.~Bahturin et al]{Yuri Bahturin}
%\thanks{Y.~Bahturin is supported by NSERC Discovery grant \#227060-19}
\address{Department of Mathematics and Statistics, Memorial University of Newfoundland, St. John's, NL, A1C5S7, Canada.}
\email{bahturin@mun.ca}

\author[]{Daniela Martinez Correa}
\thanks{D.~Correa is supported by Fapesp, grant no.~2024/01338-0 and 2025/11605-8.}
\address{Department of Mathematics, Instituto de Matem\'atica e Estat\'istica, Universidade de S\~ao Paulo, SP, Brazil}
\email{danielam.correa@ime.usp.br}

\author[]{Diogo Diniz}
\thanks{D.~Diniz was partially supported by Conselho Nacional de Desenvolvimento Cient\'ifico e Tecnol\'ogico (CNPq) grant No.~304328/2022-7.}
\address{Unidade Acadêmica de Matemática, Universidade Federal de Campina Grande, Campina Grande, PB, 58429-970, Brazil}
\email{diogo@mat.ufcg.edu.br}

%\thanks{\footnotesize $^{1}$ Partially supported by FAPESP, grant no.~2025/05699-0}

%\thanks{\footnotesize $^{2}$ Partially supported by FAPESP, grant no.~}

%\subjclass[2020]{Primary 16R10, 16W50, Secondary 20C30}

%\keywords{graded polynomial identities, finite graded algebras}

%\dedicatory{Instituto de Matemática e Estatística, Universidade de São Paulo, São Paulo,
%Brazil}

\begin{abstract}  Let $G$ be a finite group and let $\F$ be a finite field. We prove that any finite-dimensional $G$-graded  associative algebra $A$ over $\F$ has a finite basis for its $G$-graded polynomial identities.
\end{abstract}

\maketitle

\section{Introduction}

The Specht problem for an associative algebra $A$ over a field $\F$ asks if the identical relations of $A$ follow from a finite number of them.

In the case where $\F$ has characteristic zero, the Specht problem has been solved in the positive in a seminal work of A. Kemer \cite{kemer}. If $\F$ is a field of characteristic $p>0$, a negative solution was given by Kanel-Belov \cite{belov} and A. Grishin \cite{grishin}. However, it is still an open question if there exists a \emph{finite-dimensional} algebra $A$ over a  field $\F$ of  characteristic $p>0$ without finite basis of identities. In the case where $\F$ is finite, positive solutions to the Specht problem were given by R. Kruse \cite{kruse} and I. Lvov \cite{Lvov}.

A natural extension of the Specht problem is the \emph{graded Specht problem} which asks if the graded identities of an associative algebra are finitely based. In the case over zero characteristic fields a positive answer has been given in the important papers of E. Aljadeff - A. Kanel-Belov \cite{AKB} and Y. Karasik \cite{Karasik}. In the case of $\F$ of characteristic $p>0$ the answer is obviously negative in general and unknown for finite-dimensional algebras. 

In this paper we solve in the positive the Specht problem for the identical relations of \emph{group-graded finite-dimensional associative algebras over finite fields}.  

For the sake of completeness, we mention the main results on the finite basis problem in other classes of algebras. Some authors still call the finite basis problem for linear algebras over fields the Specht problem, even in the case of non-associative algebras or  algebras with additional structure.

 In 1955, Lyndon \cite{lyndon} built an example of a general universal algebra of order 6  with one binary operation that does not admit finite basis of identities. Then, 
Murski\u{i} \cite{Mursky} suggested an example of an algebra of order 3 without finite basis of identities. Later, Polin \cite{Polin} produced counter-examples for the case of finite-dimensional \emph{linear} algebras over
any field.
On the other hand, positive results have been obtained for several important classes of algebras. In particular, the problem was solved in the affirmative for finite groups, associative rings, and Lie rings in the works
%The first positive answer are given in the classes of finite group, associative rings, Lie rings in the papers 
of Oates-Powell \cite{oatespowell}, L'vov \cite{Lvov}, Kruse \cite{kruse}, Bahturin-Olshanskii \cite{Bahturinolhshanskii}. Years later, in \cite{Medvedev}, Medvedev established an affirmative solution
%solved  the problem in afirmative 
for alternative, Jordan, Malcev
and certain even more general finite algebras.

\medskip

In this work, we investigate the finite basis problem for $G$-graded associative algebras of finite dimension over finite fields. The central objective is to determine if the $T$-ideal of $G$-graded identities for such algebras is always finitely generated as a $T$-ideal, when the group $G$ is finite. Many of the methods employed in the present paper are taken from the paper \cite{Lvov}. It should be emphasized that the tools used in all papers on the Specht problem for finite rings/algebras have their origin in Group Theory \cite{oatespowell}. A convenient reference is  \cite[Chapter 5]{Neumann}. 

\medskip

The paper is organized as follows: Section 2 is devoted to the preliminaries of graded algebras and graded polynomial identities. Here we also discuss the graded Jacobson radical of a graded associative algebra. The main result of this section, given in Theorem \ref{WAgraded},  is an analogue of the classical Wedderburn Theorem. In Section 3, we recall some well-known results about Cross varieties, the main result of this section being Proposition \ref{criticalsimple}. Finally, Section 4 contains the proof of our main result, Theorem \ref{teoprincipal}, preceded by some auxiliary technical results.

\section{Preliminaries}

\subsection{Graded algebras and graded modules}
	In this section, we fix a field $\F$ and a group $G$ with identity element $e$. All algebras and vector spaces are assumed to be over $\F$. The aim of this section is to prove some properties of the graded Jacobson radical of a graded algebra. We start with some basic definitions about $G$-graded modules.
	
	\medskip

	A {\it $G$-grading} on a vector space $V$ is a decomposition 
	\(V=\bigoplus_{g\in G} V_g,\)
	where $V_g$ are subspaces of $V$, for $g\in G$. A vector space $V$ with a fixed grading is called a {\it $G$-graded vector space}. A subspace $W\subset V$ is said to be a {\it $G$-graded subspace} if
	\(W=\bigoplus_{g\in G} (V_g\cap W).\)
	
	\medskip
	
	Consider $V$ and $W$ two $G$-graded vector spaces. A linear map
	$f : V\rightarrow W$ is be called a {\it homomorphism of $G$-graded spaces} if for all $g\in G$, we have $f(V_g)\subseteq W_g$. The set of all such maps will be denoted $\textrm{Hom}^G(V,W)$. Recall that an associative algebra $\mathcal{A}$ is called a {\it $G$-graded algebra}, if it is a $G$-graded vector space such that $\mathcal{A}_g \mathcal{A}_h\subseteq \mathcal{A}_{gh}$, for $g,h\in G$. Thus, a homomorphism of $G$-graded algebras is a homomorphism of algebras such that it is also a homomorphism of $G$-graded spaces.
	
	\medskip
	
	Now, we recall some basic definitions about $G$-graded modules. Let $\mathcal{A}$ be a $G$--graded algebra. A \emph{$G$-graded right $\mathcal{A}$--module} is a right $\mathcal{A}$-module with a $G$-grading, as a vector space, such that $M_gA_h\subseteq M_{gh}$ for all $g,h\in G$. A \textit{$G$-graded right $\mathcal{A}$-submodule} of $M$ is a right $\mathcal{A}$--submodule of $M$ such that is also a $G$-graded subspace. If $M\mathcal{A}\neq \{ 0\}$ and $M$ does not admit a non-zero proper $G$-graded right $\mathcal{A}$-submodules, then $M$ is called {\it $G$-graded simple}. A {\it $G$-graded homomorphism} between $G$-graded modules is a homomorphim of right $\mathcal{A}$-modules which is also a homomorphism of $G$-graded vector spaces.
	
	\medskip
	
	Let $M$ be a $G$-graded right $\mathcal{A}$-module and let $N_1,N_2$ be $G$-graded right $\mathcal{A}$-submodules of $M$. If $N_1\subseteq N_2$ and $N_2/N_1$ is  simple as a $G$-graded module, then $N_2/N_1$ is called a {\it $G$-chief factor of $M$}.  A {\it composition series of $M$} is a finite ascending chain of $G$-submodules of $M$
	\[\{ 0\}=N_0\subseteq N_1\subseteq N_2\subseteq\cdots\subseteq N_k=M\]
	such that  each $N_{i}/N_{i-1}$ is a $G$-chief factor of $M$, for $i\in\{1,\ldots,k\}$. The following result is the  Jordan--H\"older Theorem for graded modules. It is a particular case of the Jordan - H\"older theorem for $\Omega$-algebras, see e.g. \cite{AGK}.

	\begin{theorem}[Graded Jordan--H\"older]\label{GJH}
		Let $M$ be a $G$-graded left $\mathcal{A}$--module admitting a graded composition series. Then any two graded composition series of $M$ have the same length, and their graded chief factors are isomorphic up to permutation.
	\end{theorem}

\subsection{Graded Jacobson radical}\label{ssGJR}	
	We start with recalling few definitions and results related to the notion of Jacobson radical.

\medskip
		Given an algebra $\cA$ and  a  right $\mathcal{A}$-module $M$,  the annihilator of $M$ in $\mathcal{A}$ is defined by
		\[\Ann_G(M)=\{a\in\mathcal{A}\mid xa=\{ 0\}\, \forall x\in M\}.\]

\medskip
	
	The intersection $J_G(\mathcal{A})=\bigcap \Ann_G(M)$ of the annihilators of all $G$-graded simple $\cA$- modules is called the \emph{graded Jacobson radical of $\cA$}. Any $G$-graded algebra $\cA$ is a right (left) $G$-graded $\cA$-module, denoted by $\cA_\cA$. If $K$ and $L$ are graded right ideals of $\cA$ such that  $K/L$ is a chief factor of $\cA_\cA$ then $KJ_G(\cA)\subseteq L$.
	
\medskip
 In the remainder of this section, we adapt the results about the graded Jacobson radical, given in various sources in the case of unital graded associative algebras to the case of graded algebras which are not necessarily unital.
 
 It is natural to restrict oneself to Artinian graded associative algebras. Recall that a $G$-graded algebra is right \emph{$G$-Artinian}, if every descending chain of $G$-graded right ideals of $\mathcal{A}$ has only finitely many pairwise different terms.

	%%%%%%% NILP GRADED J
	\begin{proposition}\label{nilpJGA}
		If $\mathcal{A}$ is a right $G$-Artinian algebra, then $J_G(\mathcal{A})$ is a nilpotent $G$-graded ideal.  Moreover $J_G(\mathcal{A})$ is the largest nilpotent $G$-graded ideal of $\mathcal{A}$.
	\end{proposition}
	\begin{proof}
		We first prove that $J=J_G(\cA)$ is a graded nilpotent ideal of $\cA$. Consider the descending chain of  $G$-graded right ideals:
		\[J\supseteq J^2\supseteq\cdots\supseteq J^n\supseteq \cdots.\]
		Since $\mathcal{A}$ is right Artinian, as a $G$-graded algebra,  there exists $n\geq 1$, such that
		\[J^{n}=J^{k}, \quad \forall k\geq n.\]
		%Thus, if $x\in\mathcal{A}$ and $xJ^{n+2}=\{ 0\}$, then $xJ^{n}=\{ 0\}$. \\
		%\indent
		We will show that $J^{n}=\{ 0\}$.  Consider the 2-sided graded ideal $S=\{x\in J\mid xJ^n=\{ 0\}\}$.  If $J^n\subseteq S$, then $J^nJ^n= J^{2n}=\{ 0\}$, so that $J^n=\{ 0\}$, as needed.

		If $ J^n \nsubseteq S$, consider the $G$-graded Artinian $\cA$-module $\cA/S$. The submodule $J^n+S/S$ is nonzero so contains a minimal nonzero right graded submodule $I/S$, for some graded right ideal $I$ of $\cA$.  Since $I/S$ is a simple graded $\cA$-module, $IJ_G(\cA)\subseteq S$. In particular, $IJ^n\subseteq S$. But then $(IJ^n)J^n=\{ 0\}$. Since $J^{2n}=J^n$ we have that $IJ^n=\{ 0\}$  so that $I\subset S$, which is a contradiction.
		
		To prove that $J_G(\cA)$ is the greatest graded nilpotent ideal of $\cA$, we consider a $G$-graded ideal $N$ of $\mathcal{A}$, such that $N^k=\{ 0\}$, for some natural $k$. Let $M$ be a graded simple right $\mathcal{A}$-module. If $MN\neq \{0\}$ then $M =MN$ and by induction $M=MN^k=\{0\}$, a contradiction. So $MN=\{0\}$ for any simple $G$-graded simple module $M$. It follows that $N\subseteq J_G(\mathcal{A})$, as needed.
	\end{proof}
	
		\begin{lemma}
		Let $\mathcal{A}$ be a graded algebra. Then $J_G(\mathcal{A}/J_G(\mathcal{A}))=\{ 0\}$.
	\end{lemma}
	\begin{proof}
		Let $a+J_G(\mathcal{A})$ be an element of $J_G(\mathcal{A}/J_G(\mathcal{A}))$ and let $M$ be a graded simple $\mathcal{A}$-module. Then $M$ is a graded simple $\mathcal{A}/J_G(\mathcal{A})$-module, therefore \[0=m(a+J_G(\mathcal{A}))m=ma,\] for all $m\in M$. As a consequence $a\in \mathrm{Ann}_G(M)$. Since $M$ is an arbitrary graded simple module, we conclude that $a\in J_G(\mathcal{A})$. Thus $J_G(\mathcal{A}/J_G(\mathcal{A}))=\{ 0\}$.
	\end{proof}
	
	\begin{lemma}\label{radicalideals}
		Let $\mathcal{A}$ be a finite dimensional $G$-graded algebra and $I$ a $G$-graded 2-sided ideal of $\mathcal{A}$. Then $J_G(I)=J_G(\mathcal{A})\cap I.$
	\end{lemma}
	\begin{proof}
		Since $J_G(\mathcal{A})\cap I$ is a graded nilpotent ideal of $I$,  Theorem \ref{nilpJGA} yields $J_G(\mathcal{A})\cap I\subseteq J_G(I)$. Moreover, by Theorem \ref{nilpJGA}, $J_G(I)$ is a nilpotent $G$-graded ideal of $I$. Now let $M$ be the 2-sided graded ideal of $\mathcal{A}$ generated by $J_G(I)$. Then $M$ is nilpotent. Indeed, if $j_1,j_2\in J_G(I)$ and $a_1, a_2\in \cA$ then $a_1j_1a_2j_2\in J_G(I)\cap M^2$. Therefore, if $J_G(I)^n=0$ then $M^{2n}=0$. The second statement in Theorem \ref{nilpJGA} implies that $J_G(I)\subseteq M\subseteq J_G(\mathcal{A})$. Since $J_G(I)\subseteq I$, we conclude that $J_G(I)\subseteq J_G(\mathcal{A})\cap I$.
	\end{proof}

	Let $\mathcal{A}$ be a $G$-graded $\mathbb{F}$-algebra. The \emph{unitalization} of $\cA$ is the vector space $A^{\sharp}:=\mathbb{F}\times A$ with the multiplication \[(\lambda, a)(\mu, b)=(\lambda \mu, \mu a+\lambda b+ab),\] for all $\lambda,\mu \in \mathbb{F}$ and all $a, b \in \mathcal{A}$. The unit of the algebra is $(1,0)$, henceforth we identify $\mathbb{F}$ with $\mathbb{F}(1,0)$. We also identify $a$ with the element $(0,a)$ of $A^{\sharp}$. Then we write $A^{\sharp}=\mathbb{F}\oplus A$. Note that $A^{\sharp}=\oplus_{g\in G}A^{\sharp}_g$, where $A^{\sharp}_e=\mathbb{F}\oplus A_e$ and  $A^{\sharp}_g= A_g$ if $g\neq e$, is a $G$-grading on $A^{\sharp}$. Henceforth we consider $A^{\sharp}$ with this grading.

	\begin{proposition}
		If $\mathcal{A}^{\sharp}$ is the unitalization of the finite dimensional $G$-graded algebra $\mathcal{A}$ then $J_G(\mathcal{A}^{\sharp})=J_G(\mathcal{A})$.
	\end{proposition}
	\begin{proof}
		By Proposition \ref{nilpJGA},  $J_G(A^{\sharp})$ is a nilpotent ideal of $A$, as a consequence $J_G(A^{\sharp})\subseteq A$. Since $\mathcal{A}$ is a graded ideal of $A^{\sharp}$ it follows from Lemma \ref{radicalideals} that \[J_G(A)=J_G(A^{\sharp})\cap A=J_G(A^{\sharp}).\]
	\end{proof}
	
	The previous proposition and \cite[Lemma 2.3]{DDK} imply the following corollary.
	
	\begin{corollary}
		Let $\mathcal{A}$ be a $G$-graded finite dimensional  algebra such that $J_G(\mathcal{A})=\{ 0\}$. Then $\cA$ is a unital algebra.
	\end{corollary}
	
	Recall that $\mathcal{A}$ is called {\it $G$-simple}
	if $\mathcal{A}^2\neq \{ 0\}$ and the only $G$-graded  ideals of $\mathcal{A}$ are $\{ 0\}$ and $\mathcal{A}$. Thanks to \cite{DDK}, we have the following.
	\begin{theorem}\label{WAgraded}
		Let $\mathcal{A}$ be a $G$-graded finite dimensional algebra such that $J_G(A)=\{ 0\}$. Then $\mathcal{A}$ is the direct sum of finite-dimensional $G$-simple algebras.\hfill $\Box$
	\end{theorem}
	%\begin{proof}
	%Let $I_0$ be a minimal $G$-graded ideal of $\mathcal{A}$, we claim that $I_0$ is a $G$-simple algebra. Suppose that there is a $G$-graded ideal  $B\neq 0$ of $I_0$,
	%then $I_0BI_0\subseteq B$ is a $G$-graded ideal of $\mathcal{A}$. Note that $I_0B$ is a $G$-graded left ideal of $\mathcal{A}$. By Corollary \ref{unidade1}, $I_0$ has a unit, thus $I_0B\neq \{ 0\}$, Since $J_G(A)=\{ 0\}$, $(I_0B)^2\neq \{ 0\}$, hence $I_0BI_0\neq 0$. Then, by the minimality of $I_0$, we get $I_0BI_0=I_0$, thus $I_0=B$. Hence, $I_0$ is a $G$-simple algebra.\\
	%\indent
	%Now, write $\mathcal{A}=I_0\oplus T_0$, where $T_0$ is a $G$-graded ideal of $\mathcal{A}$. By Proposition \ref{artinianideals} and Remark \ref{artinianideals}, $T_0$ is artininian and $J_G(T_0)=0$. Thus,
	%we can write $T_0=I_1\oplus T_1$, where $I_1$ is a $G$-graded (simple) ideal of $\mathcal{A}$ and $T_1$ is a $G$-graded ideal of $\mathcal{A}$. Reapeting the process, we construct a descending chain of ideals
	%\[T_0\supseteq T_1\supseteq\cdots \supseteq T_k\supseteq\cdots\]
	%of $G$-graded ideals. As $\mathcal{A}$ is $G$-graded Artinian, we have that this chain becomes stationary. Then, there is $s\geq 0$, such that $T_s=T_k$ for $k\geq s$. This implies that $\mathcal{A}=I_0\oplus I_1\oplus\cdots \oplus I_s$ and the result follows.
	%\end{proof}

	The structure of $G$-graded semisimple algebras becomes completely clear if one applies the following.
	
	\begin{theorem}\cite[Theorem 2.6]{EK}\label{Gradedsimplealgebra}
		Let $\mathcal{A}$ be a $G$-graded algebra. If $\mathcal{A}$ is $G$-simple and satisfies the descending chain condition on $G$-graded right ideals, then there exists a $G$-graded algebra $\mathcal{D}$ and a graded left right $\mathcal{D}$-module $V$ such that $\mathcal{D}$ is a graded division algebra, $V$ is finite-dimensional over $\mathcal{D}$ and $\mathcal{A}$ is isomorphic to $\textrm{End}_{\mathcal{D}}(V)$.
	\end{theorem}
\begin{remark}\label{inducedgrading}
Using the notation of Theorem \ref{Gradedsimplealgebra} and fixing $\{v_1,\ldots, v_n\}$ a homogeneous $\mathcal{D}$-basis of $V$. Suppose that $g_i$ is degree of $v_i$. Then, we can identify  $\textrm{End}_{\mathcal{\mathcal{D}}}(V)$ with $M_n(\mathcal{D})$ in natural way. Moreover, it is convenient identify $M_n(\mathcal{D})$ with the tensor product $M_n(\F)\otimes\mathcal{D}$ via Kronecker
product, this is, the element $(\lambda_{ij})\otimes d\in M_n(\F)\otimes\mathcal{D}$ with the element $(\lambda_{ij} d)\in M_n(\mathcal{D})$. Note that $G$-grading is given by
\[\deg (E_{ij}\otimes d)= g_{i} (\deg d )g_j^{-1},\]
where $d$ is a homogeneous element of $\mathcal{D}$. This grading on $M_n(\mathcal{D})$ is called the {\it $G$-grading induced} by $\mathcal{D}$. For more details, see \cite[Chapter 1]{EK} and \cite{BZ}.
\end{remark}

We conclude with a technical result, well-known for ordinary Artinian algebras.

\begin{lemma}\label{indexepimorphism}
Let $\mathcal{A}$, $\mathcal{B}$ be $G$-graded finite algebras. Suppose that there is a $G$-graded epimorphism $\psi: \mathcal{A}\rightarrow \mathcal{B}$. Then, 
\[\psi(J_G(\mathcal{A}))= J_G(\mathcal{B})\]
\end{lemma}
\begin{proof}
The result follows from  Proposition \ref{nilpJGA}, Theorem \ref{WAgraded} and the proof is analogous to that of \cite[Lemma 2.12]{Lvov}.
%\black
%First, observe that
%\[\mathcal{A}/(J_G(\mathcal{A})+ \textrm{ker}(\psi))\cong \mathcal{B}/\psi(J_G(\mathcal{A})).\]
%Moreover, $\mathcal{A}/(J_G(\mathcal{A})+ \textrm{ker}(\psi))$ is a homomorphic image of the $G$-graded algebra $\mathcal{A}/ J_G(\mathcal{A})$. In fact
%\[\mathcal{A}/(J_G(\mathcal{A})+ \textrm{ker}(\psi))\cong \dfrac{\mathcal{A}/ J_G(\mathcal{A})}{(J_G(\mathcal{A})+ \textrm{ker}(\psi))/J_G(\mathcal{A})}.\]
%By Theorem \ref{WAgraded},  $\mathcal{A}$ decomposes as the direct sum of $G$-graded simple algebras: 
%\[\mathcal{A}/ J_G(\mathcal{A})=\bigoplus_{i=1}^t \mathcal{A}_i,\]
%where all $\mathcal{A}_i$ are $G$-graded simple algebras, for $1\leq i\leq t$. Since $(J_G(A)+ \textrm{ker}(\psi))/J_G(\mathcal{A})$ is a $G$-graded ideal of $\mathcal{A}/ J_G(\mathcal{A})$, it  is the direct sum of some subset of the $G$-graded algebras $\mathcal{A}_i$. Hence,  $\mathcal{B}/\psi(J_G(\mathcal{A}))$ is direct sum of $G$-graded simple algebras. Thus, $\mathcal{B}/\psi(J_G(\mathcal{A}))$ does not contain $G$-graded nilpotent ideals. Then, $J_G(\mathcal{A})$ is mapped to zero under the natural epimorphism $\pi: \mathcal{B}\rightarrow \mathcal{B}/\psi(J_G(\mathcal{A}))$. Hence $J_G(\mathcal{B})\subseteq \psi(J_G(\mathcal{A}))$. On the other hand, since $\psi(J_G(\mathcal{A}))$ is a $G$-graded nilpotent ideal of $\mathcal{B}$,  so that the opposite inclusion holds. 
\end{proof}
%We highlight that Lemma \ref{indexepimorphism} remains valid for any finite algebras $\mathcal{A}$ and $\mathcal{B}$ in a variety admitting largest nilpotent ideals $N$ and $M$ of $\mathcal{A}$ and $\mathcal{B}$, respectively, such that the quotients $\mathcal{A}/N$ and $\mathcal{B}/M$ are direct sums of simple algebras, see \cite{Lvov}.
\subsection{Graded polynomial identities}
Let $G$ be a group, for each $g\in G$, consider the infinite countable set $X_g=\{x_1^{(g)},x_2^{(g)},\ldots,\}$ and denote by $X_G=\bigcup_{g\in G} X_g$. Let $\F\langle X_G\rangle$ be the free associative $G$-graded algebra, freely generated by $X_G$. Recall that $\F\langle X_G\rangle$ is the free associative algebra, endowed with the $G$-grading
\[\F\langle X_G\rangle=\bigoplus_{g\in G} (\F\langle X_G\rangle)_g,\]
satisfying the following condition: $\textrm{deg}_G(x_i^{(g)})= g$, for all $i\geq 1$ and $g\in G$.

A $T_G$-{\it ideal} of $\F\langle X_G\rangle$ is a $G$-graded ideal of $\F\langle X_G\rangle$ closed under all $G$-graded endomorphisms of $\F\langle X_G\rangle$. Given $S\subseteq \F\langle X_G\rangle$, we denote by $\langle S\rangle_{T_G}$, the $T_G$-ideal generated by $S$, this is the smallest $T_G$-ideal that contains $S$. If $f\in\langle S\rangle_{T_G}$, we say that $f$ is a {\it consequence} of  the polynomials in $S$.
\begin{definition}
Let $J$ be a $T_G$-ideal and consider  $S\subseteq J$. If $J= \langle S\rangle_{T_G}$, we say that $S$ is a basis of $J$ as $T_G$-ideal. 
\end{definition}
Note that we do not require $S$ to be minimal; thus the whole $J$ is a basis of $J$. %Clearly such a basis is of little value and does not contribute much to our knowledge; we are interested in ``small'' sets $S$.  
If the basis $S$ is finite, we say that the $T_G$-ideal $J$ satisfies {\it the finite basis property (f.b.p)}.\\
\indent
We recall that a polynomial $f(x_1^{(g_1)},\ldots,x_k^{(g_k)})\in \F\langle X_G\rangle$ is 
{\it$G$-graded polynomial identity} for a $G$-graded algebra $\mathcal{A}$, if
$f(a_1,\ldots, a_k)=0$, for any $a_i\in\mathcal{A}_{g_i}$ and $1\leq i\leq k$. We say that $\mathcal{A}$ is a {\it graded PI-algebra}, if it satisfies a non-trivial $G$-graded polynomial identity.

The set of all $G$-graded polynomial identitities of $\mathcal{A}$ is a $T_G$-ideal, denoted by $T_G(\mathcal{A})$. Thus, a variety of $G$-graded algebras is the class of algebras defined by a set of $G$-graded polynomial identities. Given a $G$-graded PI-algebra $\mathcal{A}$, we denote by $\var(\mathcal{A})$, the variety $\V$ of $G$-graded algebras defined by $T_G(\mathcal{A})$. In this case, $\V$ is called \emph{finitely generated}.
\begin{definition}
Let $\V$ be a variety of $G$-graded algebras and let $T_G(\V)=\bigcap_{\mathcal{A}\in\V} T_G(A)$. We say that $\V$ is finitely based, if $T_G(\V)$ satisfies f.b.p., that is, $\V$ can be defined by a finite number of $G$-graded polynomial identities. A finitely based variety $\V$ of $G$-graded algebras is called \emph{Specht}  if each of its subvarieties is finitely based.
\end{definition}
In particular, note that $\var(\mathcal{A})$ is finitely based, if $T_G(\mathcal{A})$ admits a finite basis as $T_G$-ideal. 
\begin{definition}
Let $\V$ be a variety of $G$-graded algebras. The 
\emph{$G$-graded relatively free algebra} $F_d(\V)$ of rank $d\,|G|\,$ in $\V$ 
is defined as follows. Let $X_{G,d}$ be $G$-graded set
\[
X_{G,d}=\{\, x_1^{(g)},\ldots,x_d^{(g)} \mid g\in G \,\},
\]
where each  $x_i^{(g)}$ is homogeneous of degree $g$. Then,
\[F_d(\V)=\F\langle X_{G,d}\rangle
\Big/\big(T_G(\V) \cap \F\langle X_{G,d}\rangle\big).
\]
where $T_G(\V)$ denotes the $G$-graded $T$-ideal of graded polynomial 
identities of $\V$.
\end{definition}

Every algebra in $\V$ which can be generated by $d$-graded elements is a factor-algebra of $F_d(\V)$.

\section{Cross Varieties}
We fix a set $\Omega=\bigcup_{m=0}^{\infty} \Omega_m\neq{\varnothing}$. A set $\mathcal{A}$ is called a {\it universal algebra} if every $\omega\in\Omega_{m}$ defines an $m$-ary operation on $\mathcal{A}$, that is, a map $w:\underbrace{\mathcal{A}\times\cdots\times\mathcal{A}}_m\mapsto \mathcal{A}$. The set $\Omega$ is called a {\it signature}. If we require that $\mathcal{A}$ has the structure of a vector space and that the operations are $m$-linear for each $\omega \in \Omega_m$, then $\mathcal{A}$ is called a {\it linear  $\Omega$-algebra}.

\medskip

In \cite{BY}, it was noticed that every $G$-graded algebra is a universal algebra. Indeed, let $\F$ be a field, $G$ any group and $\mathcal{A}$ an algebra over $\F$, in the usual sense, that is, with one binary linear operation $\mu$. It is not necessary that $\mathcal{A}$ is associative or has finite dimension. Consider a $G$-grading on $\mathcal{A}$,
\(\mathcal{A}=\bigoplus_{g\in G} \mathcal{A}_g.\)
Then, for any $g\in G$, we have the natural projection $\pi_g:\mathcal{A}\rightarrow\mathcal{A}$ given by
\begin{equation}\label{Gprojection}
\pi_g\left(\sum_{h\in G}a_h\right)=a_g,
\end{equation}
where $a_h\in\mathcal{A}_h$, for any $h\in G$. Moreover, for any $\lambda\in\F$, we can define $f_{\lambda}: \mathcal{A}\rightarrow \mathcal{A}$, given by $f_{\lambda}(a)=\lambda a$ for all $a\in\mathcal{A}$. Thus, consider the set $\Omega=\Omega_1\cup\Omega_2$, where $\Omega_1=\{\pi_g, f_{\lambda}\mid g\in G,\, \lambda\in\F\}$, $\Omega_2=\{\mu\,, +\}$. In this way, we view $\mathcal{A}$ as a universal algebra, where, $\mu$ and $+$ are binary operations, the original product and the sum on $\mathcal{A}$, respectively. Each $\pi_g$ is a unary operation given by (\ref{Gprojection}) and any $f_{\lambda}$ is also unary operation. In particular, any $G$-graded algebra is a $\Omega$-algebra.

\medskip

In this section we mostly translate some definitions and results from the seminal paper of Lvov \cite{Lvov}, which deals with the varieties of  finite associative algebras, into the language of \emph{graded} associative algebras. A large initial part of that paper is devoted to much more general finite algebras. As a matter of fact, the author suggests that the word \emph{algebra} will be used for $\Omega$-algebras. To avoid possible confusion, we decided to list out Lvov's result, which are specifically true for graded algebras.

\medskip

%So we fix a finite field $\F$ and a finite group $G$. All the algebras will be considered over $\F$. The aim of this section is to present some definitions and remarks which are of general use in the theory of varieties of universal algebras, formulated in the setting of $G$-graded algebras. For more details, see \cite[Section 1]{Lvov}.

 In what follows, will be dealing with the varieties $\V$ of graded associative algebras in which every finitely generated $G$-graded algebra is finite. Such varieties are called \emph{locally finite}. In \cite[Lemma 1.1]{Lvov} it is noted that every locally finite variety of $G$-graded algebras is generated by its finite $G$-graded algebras. Conversely, a variety generated by one or finitely many finite algebras is locally finite. These facts immediately follow from Birkhoff's Theorem \cite{Birk}.
    
    To minimize the generating set of a variety, one introduces the notion of a \emph{critical algebra}. Given a $G$-graded algebra $\mathcal{A}$, a \emph{graded section} of $\mathcal{A}$ is a quotient $\mathcal{B}/\mathcal{C}$, where $\mathcal{B}$ is a $G$-graded subalgebra of $\mathcal{A}$ and $\mathcal{C}$ is a $G$-graded ideal of $\mathcal{B}$.  The graded section is \emph{proper} unless $\mathcal{B}=\mathcal{A}$ and $\mathcal{C}=0$. A $G$-graded algebra $\mathcal{A}$ is \emph{critical} if it is finite and does not belong to the variety generated by its proper $G$-graded sections. Every critical algebra $\cA$ is \emph{monolithic} in the sense that the intersection $M(\cA)$ of all nonzero ideals of $\cA$ is nonzero. One calls $M(\cA)$ the monolith of $\cA$. 

Basic general properties of critical $\Omega$-algebras are proven in \cite[$\S\,1$]{Lvov}. In the particular case of graded algebras, we have the following.
 
\begin{proposition}\label{locallyfinite2}
Let $\V$ be a variety of $G$-graded algebras. The following statements are true:
\begin{enumerate} 
  \item[(i)] If $\V$ is locally finite, then it is generated by its own critical algebras.
  \item[(ii)] If  $\V$ is locally finite variety and let $\V'$ be a proper subvariety. Then there is a critical algebra $\mathcal{A}\in\V$ that does not belong to $\V'$.
  \item[(iii)] If $\V$ is a locally finite variety of $G$-graded algebras with exactly $n$ critical $G$-graded algebras. Then $\V$ has only finitely many subvarieties, and in fact their number does not exceed $2^n$.
  \item[(iv)] If all the $s$-generated $G$-graded algebras in $\V$ are finite, then there are only a finite number of such algebras in $\V$ (up to isomorphism).
  \item[(v)] If $\V$ is a locally finite variety of $G$-graded algebras. Then, $\V$ has only a finite number of subvarieties if and only if it has a finite number of $G$-graded critical algebras.
  \item[(vi)]  If a variety $\V$ of $G$-graded algebras is finitely based and contains only a finite number of subvarieties, then $\V$ is a Specht variety.
  \end{enumerate}   
\end{proposition}
An important definition, which appears in all papers on f.b.p. in finite algebras is this.
\begin{definition}
Let $\V$ be a variety of $G$-graded algebras. We say that $\V$ is a \emph{Cross variety} if it satisfies the following conditions:
\begin{itemize}
    \item[(i)] $\V$ is locally finite;
    \item[(ii)] $\V$ admits a finite basis of $G$-graded polynomial identities;
    \item[(iii)] There are only a finite number of non-isomorphic $G$-graded critical algebras. 
\end{itemize}
\end{definition}

\begin{theorem}\label{crosssubvariety}
Let $\V$ be a  Cross variety of $G$-graded algebras. Then $\V$ is a Specht variety. Moreover, any subvariety of $\V$ is also Cross.
\end{theorem}

This theorem provides an important tool for the proof of f.b.p. in a finite graded algebra.
\begin{remark}\label{identitiescross}
Let $\mathcal{A}$ be a $G$-graded algebra. Suppose that there are $G$-graded identities $f_1,\ldots, f_k$ of $\mathcal{A}$ such that the variety  $\V$ of $G$-graded algebras defined by them is locally finite and  has finite number of non-isomorphic critical algebras. Therefore $\V$ is a Cross variety, then by Theorem \ref{crosssubvariety}, $\var(\mathcal{A})$ is also Cross. Hence, $\mathcal{A}$ admits a finite basis of $G$-graded identities.
\end{remark}
Let $\mathcal{A}$ be a $G$-graded algebra. If $\{I_i\mid i\in\Lambda\}$ is a set of nonzero $G$-graded ideals such that 
$\bigcap_{i\in\Lambda} I_i=\{ 0\}$, then $\mathcal{A}$ is isomorphic to a subdirect product of its proper sections $\mathcal{A}/I_i$, $i\in\Lambda$. Hence the following is true.
\begin{proposition}
A $G$-graded critical algebra is subdirectly irreducible as a $G$-graded algebra.
\end{proposition}
As consequence of \cite[Theorem 8]{BDMY} and Theorem \ref{WAgraded}, we have the following 
\begin{proposition}\label{criticalsimple}
Let $\mathcal{A}$ be a finite $G$-graded algebra such that $J_G(A)=\{ 0\}$. Then $\mathcal{A}$ is critical, as $G$-graded algebra, if and only if, $\mathcal{A}$ is $G$-graded simple.
\end{proposition}

\section{Main Result}
In this section, we will prove the main theorem of this work, that is, if $G$ is a finite group, then every finite $G$-graded associative algebra $\mathcal{A}$ has finite basis of $G$-graded identities. By Remark \ref{identitiescross}, it is sufficient to find a finite subset of $G$-graded identities of $\mathcal{A}$ such that their defined variety is locally finite and contains only a finite number of $G$-graded critical algebras. 

\begin{remark}\label{gradedgenerators}
Let $\mathcal{A}$ be a $G$-graded algebra, suppose that $\mathcal{A}$ is finitely generated, as an algebra. Then there exits a finite generator set $S$, suppose that $\mid S\mid=d$. Then, the set $S'$ of the homogeneous components of the elements of $S$ also generates $\mathcal{A}$. Note that $S'$ contains at most $d$ elments of degree $g$, for each $g\in G$. Thus, $\mid S'\mid\leq \mid G\mid d$.  
%In this case, we say that $\mathcal{A}$ is $d$-generated as a $G$-graded algebra.
\end{remark}

We start with a criterion for the local finiteness of a variety of $G$-graded algebras.
%{\color{blue}
\begin{proposition}\label{criteriolocallyfinite}
Let $G$ be a finite group and let $\V$ be a variety of $G$-graded algebras. Then, the following conditions are equivalent:
\begin{itemize}
    \item[(i)] $\V$  is locally finite;
    %\item[(ii)] all the $1$-generated algebras of $\V$ are finite;
    %\item[ii')] The relatively free algebra $F_1(\V)$  of rank $1$ is finite.
    \item[(ii)]  the following identity holds in $\V$:
    \begin{equation}\label{4.1(ii)}
    (x_1^{(e)})^k=(x_1^{(e)})^l,\,\,\, k>l>0 \mbox{ integral };
    \end{equation}
%\item[iii')]
%There exists a positive integer $k$ such that, for any $(g_1,\ldots, g_k)\in G^k$,  the variety $\V$ satisfy the following $G$-graded identity:
%    \[ x_1^{(g_1)}\cdots x_k^{(g_k)}= f(x_1^{(g_1)},\ldots, x_k^{(g_k)}), \]
%where $f(x_1^{(g_1)},\ldots, x_k^{(g_k)})$ is a $G$-graded polynomial with ordinary degree at least $k+1$
\item[(iii)] the following identity holds in $\V$:
\begin{equation}\label{4.1(iii)}
\sum_{i=s}^r c_i (x_1^{(e)})^i,  
\end{equation}
where $r\geq s>0$ and $c_r\neq 0$.
\end{itemize}
\end{proposition}
\begin{proof}
Assume that $\V$  is locally finite. Then $F_1(\V)=\F\langle X_{G,1}\rangle/ (\F\langle X_{G,1}\rangle\cap T_G(\V))$ is a finite algebra, therefore there exist  $k>l>0$ such that the equality (\ref{4.1(ii)}) holds. Hence (i) $\Rightarrow$ (ii). It is clear that (ii) $\Rightarrow$ (iii). 

Now assume that (iii) holds. If $F_d(\V)=\F\langle X_{G,d}\rangle/(\F\langle X_{G,d}\rangle\cap T_G(\V)) $, then (\ref{4.1(iii)}) implies that $(F_d(\V))_e$ is a PI-algebra. Since $G$ is a finite group, the main result of \cite{BGR} implies that $F_d(\V)$ is a PI algebra. Then Shirshov’s Height Theorem (see \cite{Shirshov}) implies that there exists $m, l$ such that $F_d(\V)$ is spanned as a vector space by the products 
\begin{equation}\label{gen}
u_{i_1}^{k_1}\cdots u_{i_t}^{k_t},
\end{equation}
where $t\leq m$ and $u_{i_1},\dots, u_{i_t}$ lie in the set $\mathcal{U}$ of monomials in $X_{G,d}$ of degree $\leq l$. Note that the $|\,G\,|$th power of any homogeneous element lies in $(F_d(\V))_e$, hence (\ref{4.1(iii)}) implies that every homogeneous element in $F_d(\V)$, in particular each monomial in $\mathcal{U}$, is algebraic of degree  at most $\mid G\mid r$.
%since this set is finite there exists a natural number $\kappa$ such that the monomials of $\mathcal{U}$ are algebraic of degree $\leq \kappa$. 
Hence $F_d(\V)$ is generated, as a vector space, by the products in (\ref{gen}), such that $k_i\leq \mid G\mid r$ for all $i=1,\dots, t$. This set is finite, hence $F_d(\V)$ has finite dimension and therefore is finite. This implies that $\V$  is a locally finite variety.
\end{proof}
\begin{remark}\label{conditioncross}
Since the group $G$ is finite, one can use \cite[Lemma 2.3]{Lvov} to observe the following. If a finite $G$-graded algebra $\mathcal{A}$ has a finite set of $G$-graded identities such that every $G$-graded critical algebra satisfying these identities is $s$-generated, where $s$ is a fixed natural number, then $\var(A)$ is Cross.
\end{remark}

The following proposition provides a criterion for a $G$-graded algebra to be noncritical (as $G$-graded algebra). The criterion was originally obtained for rings by Lvov in \cite{Lvov}. We emphasize that it remains valid for $G$-graded algebras, since the notion of the free sum had been extended to $\Omega$-algebras in the paper of Kurosh \cite{AGK1}. We state the proposition in a form similar to that presented in \cite{Medvedev} by Medvedev.
%important if we want to bound the number of homogeneous generators of a $G$-graded critical algebra. 
%We use the following notation.
%\begin{notation}
%Let $A_1,\ldots, A_{n}$ be subsets of the algebra $\mathcal{A}$. Then, we denote by $\langle A_1,\ldots, A_n\rangle$, the vector space generated by all the products of the form $a_1\cdots a_n$, for $a_i\in A_i$ and also by all the products obtained by all permutations of the elements $a_i$.
%\end{notation}
\begin{proposition}\label{criticalconditions}
Let $L, M_1,\ldots, M_n$ be subsets of homogeneous elements of $\mathcal{A}$ that satisfy the followings conditions:
\begin{itemize}
\item[(i)] The set $L\cup M_1\cup\cdots \cup M_n$ generates $\mathcal{A}$, as a $G$-graded algebra;
    \item[(ii)] For any $i\in\{1,\ldots,n\}$,  $\mathcal{A}$ is not generated by $L\cup \bigcup_{j\neq i} M_j$, as a $G$-graded algebra;
    \item[(iii)] Any word $v(a_1,\ldots, a_r)$ comprised of an arbitrary number of elements from $L\cup M_1\cup\cdots \cup M_n$ is equal to zero if it contains at least one element from each set $M_i$, for $1\leq i\leq n.$
    %$\langle M_1,\ldots M_n, \underbrace{\mathcal{A}, \ldots, \mathcal{A}}_{k\, \textrm{times}}\rangle=0$, for any $k\geq 0$.
\end{itemize}
Then $\mathcal{A}$ belongs to a variety generated by proper $G$-graded subalgebras. In particular, this algebra is not critical as $G$-graded algebra.
\end{proposition}
In Proposition \ref{criticalconditions}, we would like to emphasize that the term “word” refers to a product of arbitrary length of elements in $L\cup M_1\cup\cdots \cup M_n$. The proof of this proposition is completely analogous to the proof of \cite[Lemma 2.25]{Lvov}.
The previous proposition is a useful tool, if we want to bound the number of homogeneous generators of a $G$-graded critical algebra.

\begin{corollary}\label{nilpcritical}
A $G$-graded critical nilpotent algebra $\mathcal{A}$ of index $s$ is $(s-1)$-generated.
\end{corollary}
\begin{proof}
Let $m_1,\ldots, m_k$ be a minimal system of homogeneous generators of $\mathcal{A}$. It is sufficient to prove that $k<s$. Supose that $k\geq s$ and let $L=\{ 0\}$ and $M_i=\{m_i\}$, for $1\leq i\leq k$. Notice that the conditions (i)-(ii) of Proposition \ref{criticalconditions} are satisfied by the choice of the set  $\{m_1,\ldots, m_k\}$, while the condition (iii) holds because $\mathcal{A}^k=\{ 0\}$ due to $k\geq s$. Then, by Proposition \ref{criticalconditions}, $\mathcal{A}$ is not critical. But this is an contradiction, hence the result follows.
\end{proof}
\begin{corollary}\label{nilpcritical2}
Any variety $\V$ of $G$-graded  nilpotent algebras
contains only a finite number of critical $G$-graded algebras.
\end{corollary}
\begin{proof}
Let $s$ be the nilpotency index of the free algebra of countable rank in $\V$. Then $s$ bounds the nilpotency indices of all algebras in $\V$. By Corollary \ref{nilpcritical}, we have that any nilpotent $G$-graded critical algebra in $\V$ is $(s-1)$-generated. Now, consider the relatively free algebra $F_{s-1}(\V)$ in $\V$. Since this algebra is finite and  any  $(s-1)$-generated $G$-graded subalgebra is a homomorphic images of $F_{s-1}(\V)$, there only a finite number of $(s-1)$-generated $G$-graded algebras in $\V$. In particular, this variety contains only a finite number of $G$-graded critical algebras.
\end{proof}

\begin{proposition}\label{maximalnumberofgenerators}
Let $\mathcal{A}$ be a $G$-graded critical algebra and let $I$ be a nilpotent $G$-graded ideal with nilpotent index $s$. If $\mathcal{A}/I$ is generated by $d$ homogeneous elements, then $\mathcal{A}$ is generated by no more than $(s+d-1)$ homogeneous elements.
%$d+s$-elements
\end{proposition}
\begin{proof}
Let $\overline{a}_1,\ldots,\overline{a}_d$ be homogeneous generators of $\mathcal{A}/I$. For each $i$, choose a homogeneous preimage $a_i \in \mathcal{A}$ of $\overline{a}_i$, and set $L=\{a_1,\ldots,a_d\}$. Next, choose a minimal system of homogeneous generators $b_1,\ldots,b_p$ for the ideal $I$. Let us show that $p<s$. Indeed,  assume $p\geq s$ and let $M_i=\{b_i\}$ for $1\leq i\leq p$. \\
\indent
Thus, for any word $v(a_1,\ldots, a_r)$,  comprised of an arbitrary number of elements from $L\cup M_1\cup\cdots \cup M_p$, belongs to $I^p$, if it contains at least one element from each set $M_i$, for $1\leq i\leq p$. Notice that $I^p\subseteq I^s=0$.  
%Thus, for any $k\geq 0$, we have
%\[\langle M_1,\ldots M_p, \underbrace{A,\ldots, A}_{k\,\textrm{times}}\rangle\subseteq I^p\subseteq I^s=\{ 0\}.\]
Hence, by Proposition \ref{criticalconditions}, $\mathcal{A}$ is not critical. This is a contradiction. Hence, $p\leq s-1$ and the result follows.
\end{proof}
As a direct consequence of Proposition \ref{criticalconditions}, we have the following result:
\begin{corollary}\label{corcritical}
Let $\mathcal{A}$ be a $G$-graded critical algebra and  $J$  a $G$-graded ideal of $\mathcal{A}$. Suppose that $\mathcal{A}/J$ is the direct sum of its non-zero  $G$-graded ideals $\overline{M_1},\ldots,  \overline{M_n}$. Denote by $M_1,\ldots, M_n$ the pre-imagens of these ideals in the algebra $\mathcal{A}$ under the natural projection of $\mathcal{A}$ onto the factor algebra $\mathcal{A}/J$. Let $S_n$ be the symmetric group of degree $n$. Then, there exits $\sigma\in S_n$ such that
\[M_{\sigma(1)}M_{\sigma(2)}\cdots M_{\sigma(n)}\neq 0.\]

%Then, there is a non-zero word  $v(a_1,\ldots, a_r)$  comprised elements from $M_1\cup\cdots \cup M_n$, where at least one element from each set $M_i$, for $1\leq i\leq n$ appears.
%\[\langle M_1,\ldots, M_n\rangle\neq \{ 0\}.\]
\end{corollary}
%%%%%%%%%%%%%%%%%%%%%%%%%%%%%%
\begin{proposition}\label{cotasumandosimples}
Let $\mathcal{A}$ be a $G$-graded critical algebra whose Jacobson radical $J_G(\mathcal{A})$ has nilpotent index $s$. Then the factor algebra $\mathcal{A}/J_G(\mathcal{A})$ is the direct sum of no more than $2s-1$ $G$-graded simple algebras.
\end{proposition}
\begin{proof}
By Theorem \ref{WAgraded},  we have that  $\mathcal{A}/J_G(\mathcal{A})= \bigoplus_{i=1}^ t \overline{\mathcal{A}_i}$,
where $\overline{\mathcal{A}_i}$ are $G$-graded simple algebras. Denote  by  $\mathcal{A}_i$ the preimage of $\overline{\mathcal{A}_i}$ under the natural epimorphism $\pi: \mathcal{A}\rightarrow \mathcal{A}/J_G(\mathcal{A})$, for $1\leq i\leq t$. Then
\[\overline{\mathcal{A}_i}\cdot \overline{\mathcal{A}_j}=\{ 0\}, \quad \pi(\mathcal{A}_i\mathcal{A}_j)=\{ 0\}, \quad \mathcal{A}_i\mathcal{A}_j\subseteq J_G(\mathcal{A}).\]
As the algebra $\mathcal{A}$ is associative, we have that
\[ A_{\sigma(1)}\ldots A_{\sigma(t)}\subseteq J_G(\mathcal{A})^{[\frac{t}{2}]},\]
for all $\sigma\in S_t$. By Corollary \ref{corcritical}, it follows that $J_G(\mathcal{A})^{[\frac{t}{2}]}\neq \{ 0\}$. Hence, $[\frac{t}{2}]< s$ and the result follows.
\end{proof}
\begin{theorem}\label{boundedcritical}
A $G$-graded critical algebra $\mathcal{A}$ can be generated by $(4|G| + 1)s-2|G|-1$ homogeneous elements, where $s$ is the nilpotent index of $J_G(\mathcal{A})$.
\end{theorem}
\begin{proof}
By Theorem \ref{WAgraded}, we can write $\mathcal{A}/J_G(\mathcal{A})= \bigoplus_{i=1}^ t \overline{\mathcal{A}_i}$, where $\overline{\mathcal{A}_i}$ are $G$-graded simple algebras.
Using Theorem \ref{Gradedsimplealgebra},  for $1\leq i\leq t$, we have
\[\overline{\mathcal{A}_i}\cong M_{m_i}(\mathcal{D}_i),\]
where $\mathcal{D}_i$ is a $G$-graded division algebra and the $G$-grading on $M_{m_i}(\mathcal{D}_i)$ is as in Remark \ref{inducedgrading}. By Wedderburn's Little Theorem, $\Delta_i=(\mathcal{D}_{i})_e$ is a finite field and $\F\subseteq \Delta_i$. Let $X_{i,e}$ be a  generator of the multiplicative group of $\Delta_i$; choose a nonzero homonegeous element $X_{i,g}$ in $(\mathcal{D}_i)_g$. Then,
\[X_{i,g}X_{i,h}=\sigma(g,h) X_{i,gh},\]
where $\sigma(g,h)\in\Delta_i$. Hence, $\mathcal{D}_i$ is generated, as $G$-graded algebra, by the following set of homogeneous elements
\[E_{\Delta_i}=\{X_{i,g}\mid g\in \textrm{Supp}_G(\mathcal{D}_i)\},\]
where $\textrm{Supp}_G(\mathcal{D}_i)$ denotes the support of the $G$-grading of $\mathcal{D}_i$. It follows  that $| E_{\Delta_i}| \leq |G|$. Note that $M_{m_i}(\mathcal{D}_i)$ is generated by the elements 
\[W_{i,g}= X_{i,g}e_{12},\quad Z= e_{12}+e_{23}+\cdots + e_{(m_i-1)m}+ e_{m_{i}1},\]
for $g\in G$. Therefore, this algebra can be generated by $|E_{\Delta_i}|+1$ elements. Thus, using Remark \ref{gradedgenerators}, we can consider the homogeneous components of $Z$ together with the homogeneous elements $W_{i,g}$, therefore the algebra $\mathcal{A}/J_G(\mathcal{A})$ can be generated by at most $2|G|t$ homogeneous elements. On the other hand, using Proposition \ref{cotasumandosimples}, $t\leq 2s-1$. Then, by Proposition \ref{maximalnumberofgenerators}, $\mathcal{A}$ can be generated by at most $2|G|t+s-1$ homogeneous elements. We have 
\[2|G|t+s-1\leq 2|G|(2s-1)+s-1=(4|G|+1)s-2|G|-1\]
and the result follows.
%= s(2|G|(|G|+1)+1)-|G|(|G|+1)-1
\end{proof}

%Notice that 
According to Remark \ref{conditioncross}, we can see that in order to prove that the variety generated by a $G$-graded finite algebra $\mathcal{A}$ is Cross it will be sufficient to prove that there is a finite collection of $G$-graded identities $\{f_1,\ldots, f_k\}$ of $\mathcal{A}$ such that the $G$-graded critical algebras that satisfy these identities could be generated by a bounded number of homogeneous elements. In light of Theorem \ref{boundedcritical},  it is sufficient to show that these $G$-graded critical algebras have a common bound for the nilpotency indices of their $G$-graded Jacobson radicals.\\
\indent
Observe that a variety  $\V$ of $G$-graded algebras has an infinite number of subvarieties if the indices of its nilpotent $G$-graded algebras is not bounded. In particular, $\V$ cannot be Cross. Thus, our following step will be to prove that in the variety of $G$-graded algebras generated by $\mathcal{A}$, the nilpotent $G$-graded algebras have bounded nilpotent index and then show that this property implies a finite number of $G$-graded identities of $\mathcal{A}$. Then, having this in mind, we give the following definition:

\begin{definition}
Let $\V$ be a variety of $G$-graded algebras. The nilpotency index of $\V$ is the upper bound of the index of nilpotency of its $G$-graded nilpotent algebras. 
\end{definition}
Thus, in light of of item (iii) of Corollary \ref{locallyfinite2} and the discussion above, we obtain:
\begin{proposition}\label{finiteindex}
A Cross variety of $G$-graded algebras has finite nilpotency index.
\end{proposition}
We shall prove that the converse of previous assertions holds. For this, we establish the following lemmas:

\begin{lemma}\label{criterio2}
Let $\V$ a variety of $G$-graded algebras. Then, $\V$ has nilpotency index $\leq n$ if and only if for each $\mathcal{A}\in\V$
\[\mathcal{A}^{n}=\mathcal{A}^{k},\quad \forall k\geq n.\]
\end{lemma}
\begin{proof}
First, suppose that $\V$ has nilpotency index $\leq n$ and consider $\mathcal{A}\in\V$. Thus, we have the following descending chain of ideals
\[\mathcal{A}^n\supseteq\mathcal{A}^{n+1}\supseteq\mathcal{A}^{n+2}\supseteq\mathcal{A}^{n+3}\supseteq\cdots\]
Since $\mathcal{A}/\mathcal{A}^{n+i}$ is nilpotent and belongs to $\V$, we have that
 \[(\mathcal{A}/\mathcal{A}^{n+i})^n=\mathcal{A}^n/\mathcal{A}^{n+i}=\{ 0\}.\] 
Hence, $\mathcal{A}^n\subseteq\mathcal{A}^{n+i}=\{ 0\}$ and the result follows. The converse is straightforward.
\end{proof}

\begin{lemma}\label{decompe}
Let $G$ be a group of order $n$. Then if $m> nk$ and $M$ is a monomial of degree $m$ in the free $G$-graded algebra $\F\langle X_G\rangle$ then \[M=UM_1\cdots M_kV,\] where $U,V,M_1,\dots, M_k$ are monomials in $\F\langle X_G\rangle$, $M_1,\dots, M_k$ have degree $e$ and the monomials $U$ or $V$ may be the empty.
\end{lemma}
\begin{proof}
  Let \[
M=x_{i_1}^{(g_1)}x_{i_2}^{(g_2)}\cdots x_{i_m}^{(g_m)},\] and let $p_j=g_1g_2\cdots g_j$, for $1\leq j \leq m$. Since $m>nk$ and $|G|=n$ the pigeonhole principle implies that there exist $1\leq j_1<\cdots<j_{k+1}< m$ such that \[p_{j_1}=\cdots=p_{j_{k+1}}.\] Then for $l=1,\dots, k$ the monomials \[M_l=x_{i_{{j_l}+1}}^{(g_{j_{l}+1})}\cdots x_{i_{j_{l+1}}}^{(g_{j_{l+1}})},\] have degree $p_{j_l}^{-1}p_{j_{l+1}}=e$. Hence we obtain the desired decomposition if we set $U = x_{i_1}^{(g_1)}\cdots x_{i_{j_1}}^{(g_{j_1})},$ and $V=x_{i_{j_{k+1}+1}}^{(g_{j_{k+1}+1})}\cdots x_{i_m}^{(g_m)}$.
\end{proof}

\begin{corollary}\label{identitiesfiniteindex}
Let $\V$ be a variety of $G$-graded algebras. Then, $\V$ has a bounded nilpotency index,
%$\V^{k}=\V^{k+1}$ for any $\V\in\V$, 
if and only if, $\V$ satisfies a $G$-graded identity 
\begin{equation}\label{ide}
x_{1}^{(e)}\cdots x_t^{(e)}=f(x_{1}^{(e)},\dots, x_t^{(e)}),
\end{equation}
where, $f(x_{1}^{(e)}\cdots x_t^{(e)})$ is a $G$-graded polynomial that has lowest ordinary degree $\geq t+1$.
\end{corollary}
\begin{proof}
Let $\mathcal{B}$ be a nilpotent algebra in $\V$. The identity (\ref{ide}) and \cite[Corollary 2.7]{Lvov} imply that $\mathcal{B}_e$ is nilpotent of degree $\leq t$. Lemma \ref{decompe} implies that any product $b=b_1\cdots b_m$ of $m>t|G|$ homogeneous elements can be written as \[b=um_1\cdots m_tv,\] where $m_1,\dots, m_t\in \mathcal{B}_e$. Since $(\mathcal{B}_e)^t=\{ 0\}$ we conclude that $b=0$. Thus
$\mathcal{B}$ has index $\leq t|G|$. As a consequence, $\V$ has bounded nilpotency index. For the converse, assume that $\V$ has nilpotency index $\leq t$. And let $F_t(\V)$ be the free algebra of rank $t|G|$ in $\V$. The subalgebra $\mathcal{G}$ of $F_t(\V)$ generated by $x_1^{(e)},\dots, x_t^{(e)}$ lies in $\V$, therefore $\mathcal{G}^t=\mathcal{G}^{t+1}$. Then $x_1^{(e)}\cdots x_t^{(e)}$ lies in $\mathcal{G}^{t+1}$ and there exists $f(x_{1}^{(e)}\cdots x_t^{(e)})$ with lowest ordinary degree $\geq t+1$ such that \[x_{1}^{(e)}\cdots x_t^{(e)}=f(x_{1}^{(e)},\dots, x_t^{(e)}).\] This gives the desired identity in $\V$.
\end{proof}

%%%%%%%%%%%%%%%%%%%%
\begin{corollary}\label{subvarietiefinitebased}
Let $\V$ be a variety of $G$-graded algebras such that its nilpotency index is finite. Then, there is a finitely based variety $\mathfrak{M}$ of $G$-graded algebras, with the same nilpotency index as $\V$, such 
that $\V\subseteq\mathfrak{M}$.
\end{corollary}
\begin{proof}
It is a direct consequence of Corollary \ref{identitiesfiniteindex}.
\end{proof}

\begin{theorem}\label{criteriocross}
A variety $\V$ of finite $G$-graded algebras is Cross if and only if it has  finite nilpotency index.
\end{theorem}
\begin{proof}
Due to Proposition \ref{finiteindex}, we only need to prove the converse of the theorem. Let $\V$ be a variety of $G$-graded algebras that has nilpotency index $k$. By Corollary  \ref{subvarietiefinitebased}, there is a variety $\mathfrak{M}$  of $G$-graded algebras of nilpotency index  $k$ such that is finitely based and $\V\subseteq\mathfrak{M}$. Note that $\mathfrak{M}$ is defined by the $G$-graded identity given in Corollary \ref{identitiesfiniteindex}. Then, using Proposition \ref{criteriolocallyfinite} and Theorem \ref{boundedcritical}, we have that $\mathfrak{M}$ is locally finite and has a finite number of $G$-graded critical algebras. It follows that $\mathfrak{M}$ is a Cross variety. Hence, by Theorem \ref{crosssubvariety}, $\V$ is also Cross.
\end{proof}

\begin{proposition}\label{indexoffinitealgebra}
Let $\mathcal{A}$ be a finite $G$-graded algebra. Then 
the nilpotency index of $\var(\mathcal{A})$ equals the maximum nilpotency index among the $G$-graded nilpotent subalgebras of $\mathcal{A}$.
\end{proposition}
\begin{proof}
Observe that it is sufficient to prove that the nilpotent index of $\var(\mathcal{A})$ does not exceed $s$. By Lemma \ref{criterio2}, it is enough to prove that $F_s(\var(A))^s=F_s(\var(A))^l$ for any $l>s$. Notice that this equivalent to showing that 
$F_s(\var(\mathcal{A}))/F_s(\var(\mathcal{A}))^l$ is nilpotent with index $\leq s$. Using Birkhoff's  embedding \cite{Birk} for the relatively free algebra in $\var(\mathcal{A})$, we have that 
\[F_s(\var(\mathcal{A}))\hookrightarrow \prod_{i=1}^t\mathcal{A}_i,\quad \mathcal{A}_i\cong\mathcal{A}.\]
By Theorem \ref{nilpJGA}, $J_G(F_s(\var(\mathcal{A}))$ is a nilpotent $G$-graded ideal. Thus, $J_G(F_s(\var(\mathcal{A}))$ is a subalgebra of the direct product of nilpotent $G$-graded subalgebras of $\mathcal{A}$ and, consequently, has nilpotency index $\leq s$. Now, consider the natural projection $\psi:F_s(\var(\mathcal{A}))\rightarrow  F_s(\var(\mathcal{A}))/F_s(\var(\mathcal{A}))^{l}$. Applying Lemma \ref{indexepimorphism}, we find that $F_s(\var\cA)^s\subset F_s(\var\cA)^l$, and the proof is complete. 
\end{proof}

\begin{theorem}\label{teoprincipal}
Let $\mathcal{A}$ be a finite $G$-graded algebra, then $\var(\mathcal{A})$ is a Cross variety. In particular, $T_G(\mathcal{A})$ satisfies the finite basis property.
\end{theorem}
\begin{proof}
By Proposition \ref{indexoffinitealgebra}, $\var(\mathcal{A})$ has finite nilpotency index. Hence, using Theorem \ref{criteriocross}, $\var(\mathcal{A})$ is a Cross variety. It follows that $\mathcal{A}$ admits finite basis of $G$-graded polynomial identities.
\end{proof}

\end{document}